\documentclass[12pt,reqno]{amsart}
\usepackage{graphicx}
 \usepackage{epstopdf}
\usepackage{url}
\usepackage{latexsym}
\usepackage{amsfonts,amsmath,amssymb}
\newtheorem{theorem}{Theorem}[section]

\newtheorem{prop}[theorem]{Proposition}

\newtheorem{coro}[theorem]{Corollary}

\newtheorem{remark}{Remark}

\newcommand{\qbin}[2]{\genfrac[]{0pt}{}{#1}{#2}}

\newcommand{\s}{{\sigma}}

\newcommand{\angbinom}[2]{\genfrac{\langle}{\rangle}{0pt}{}{#1}{#2}}

\def\Sym{\mathfrak{S}}

\def\Z{\mathbb{Z}}

\def\R{\mathbb{R}}
\def\N{\mathbb{N}}

\def\s{\sigma}
\def\a{\alpha}
\newcommand\myatop[2]{\genfrac{}{}{0pt}{}{#1\hfill}{#2\hfill}}

\def\des{\mathrm{des}}
\def\inv{\mathrm{inv}}

\def\lrmin{\mathrm{lrmin}}
\def\rlmin{\mathrm{rlmin}}
\def\maj{\mathrm{maj}}

\numberwithin{equation}{section}
\date{\today}
\usepackage{color}
\usepackage[dvipsnames]{xcolor}

\usepackage{tikz}

\def\and{\mathrm{d}}
\allowdisplaybreaks[4]

\begin{document}

\title[A bi-Stirling-Euler-Mahonian polynomial]{A bi-Stirling-Euler-Mahonian polynomial}
\author{Chao Xu and Jiang Zeng}
\address{ Universite Lyon 1,
UMR 5208 du CNRS, Institut Camille Jordan\\
F-69622, Villeurbanne Cedex, France}
\email{xu@math.univ-lyon1.fr}
\address{ Universite Lyon 1,
UMR 5208 du CNRS, Institut Camille Jordan\\
F-69622, Villeurbanne Cedex, France}
\email{zeng@math.univ-lyon1.fr}
\begin{abstract} 
Motivated by recent work on (re)mixed Eulerian numbers, we provide a combinatorial interpretation of a subfamily of the remixed Eulerian numbers introduced by Nadeau and Tewari. More specifically, we show that these numbers can be realized as the generating polynomials of permutations with respect to the statistics of left-to-right minima, right-to-left minima, descents, and the mixed major index. Our results generalize both the bi-Stirling-Eulerian polynomials of Carlitz-Scoville and the Stirling-Euler-Mahonian polynomials of Butler.
\end{abstract}

\maketitle
\section{Introduction}

The Eulerian numbers $  \angbinom{n}{k}$ ($n> k\geq 0$) can be  defined by the recurrence 
\begin{equation}\label{euler's recurrence}
          \angbinom{n}{k}=(n-k)  \angbinom{n-1}{k-1}+(k+1)  \angbinom{n-1}{k},
\end{equation}
with initial condition $\angbinom{0}{0} = 1$ and $\angbinom{n}{k} = 0$ if $k \notin {0,\ldots,n-1}$.
This recurrence  immediately  implies the symmetry
$  \angbinom{n}{k}=  \angbinom{n}{n-k-1}$. 
It is well-known~\cite{FS70,Pe15} that $  \angbinom{n}{k}$ counts  the number of permutations in $\Sym_n$ with $k$ descents.
The $n$th Eulerian polynomial, which generates
  the Eulerian numbers, is defined by 
\begin{equation}\label{eulerian polynomial}
    A_n(x)=\sum_{k=0}^{n-1}\angbinom{n}{k}x^k
\end{equation}
 and has  the exponential generating function 
\begin{equation}\label{exp:Euler}
   1+\sum_{n=1}^{\infty}A_n(x)\frac{t^n}{n!}= \frac{1-x}{e^{t(x-1)}-x}.
\end{equation}
 Setting $A(r,s):=\angbinom{r+s+1}{r}$, Carlitz-Scoville~\cite{CS74}  noticed   the 
 following  symmetric formula
\begin{equation}
F(x,y):=\sum_{r,s=0}^{\infty}A(r,s)\frac{x^ry^s}{(r+s+1)!}=\frac{e^x-e^y}{xe^y-ye^x},
\end{equation}
and  further introduced  the $(\alpha,\beta)$-Eulerian numbers 
    $A(r,s\,|\,\alpha, \beta)$ by
\begin{equation}\label{CS-egf}
\sum_{r,s=0}^{\infty}A(r,s\,|\,\alpha, \beta)\frac{x^ry^s}{(r+s)!}=(1+xF(x,y))^\alpha(1+yF(x,y))^\beta.
\end{equation}
Note that  $A(r,s\,|\,\alpha, \beta)=A(s,r\,|\,\beta,\alpha)$ and  $A(r,s\,|\,1,1)=A(r,s)$.
The first few terms of the above generating function  are
\begin{multline}
   1+(\alpha  x +\beta  y) +\frac{1}{2}({\alpha^{2} x^{2}}+{\left(2 \alpha  \beta +\alpha +\beta \right) y x}+{\beta^{2} y^{2}})+\frac{1}{6}\left(
   {\alpha^{3} x^{3}}+{\beta^{3} y^{3}}+\right.\\
  \left. {\left(3 \alpha^{2} \beta +3 \alpha^{2}+3 \alpha  \beta +\alpha +\beta \right) y x^{2}}+{\left(3 \alpha  \beta^{2}+3 \alpha  \beta +3 \beta^{2}+\alpha +\beta \right) x y^{2}}\right)
+\cdots.
\end{multline}
Carlitz-Scoville~\cite[\S8]{CS74}  proved the recurrence 
\begin{equation}\label{rec-CS}
A(r,s\,|\,\alpha,\beta)=(s+\alpha)A(r-1,s\,|\,\alpha,\beta)+(r+\beta)A(r,s-1\,|\,\alpha,\beta)
\end{equation}
and  also provided the  combinatorial interpretation
\begin{equation}\label{Combinatorics C-S}
A(r,s\,|\,\alpha,\beta)=  \sum\limits_{\substack{\sigma\in\Sym_{r+s+1} \\ 
\mathrm{des}(\sigma)=r}}\alpha^{\lrmin(\sigma)-1}
   \beta^{\rlmin(\sigma)-1},
\end{equation}
where $\des(\sigma)$ is the number of descents of $\sigma$, while $\lrmin(\sigma)$ (resp. $\rlmin(\sigma)$) is the number of \emph{left-to-right}  (resp. 
\emph{right-to-left}) minima of $\sigma$, see \eqref{def:statistics}.

 Let $\lambda_1\geq \cdots\geq \lambda_{r+1}$ be real numbers and let $\lambda=(\lambda_1, \ldots, \lambda_{r+1})$. The permutahedron $\rm{Perm}(\lambda)$ is the convex hull of points $\lambda_\s=(\lambda_{\s(1)}, \ldots, \lambda_{\s(r+1)})$ where $\s$ is a permutation in $\Sym_{r+1}$. It is known that 
 \[
 \rm{vol}(\rm{Perm}(\lambda))=\sum_{c=(c_1, \ldots, c_r)}A_c\frac{\mu_1^{c_1}\cdots \mu_r^{c_r}}{c_1!\cdots c_r!}
 \]
 where $c\in \mathcal{W}_r=\{(c_1, \ldots, c_r): c_1+\cdots +c_r=r\}$ and $\mu_i=\lambda_i-\lambda_{i+1}$.  The coefficients $A_c$ are positive integers known as the \emph{mixed Eulerian numbers}~\cite{Po09}.

In this paper we adopt  the standard $q$-notations~\cite{An76}. For  $x\in \R$ the $q$-real number $[x]$ is defined by 
\[    
[x]=\frac{1-q^{x}}{1-q}.
\]
For an integer  $r\in\N$, we define
$[r]!=1\cdot (1+q)\cdots (1+q+\cdots +q^{r-1})$ and 
the $q$-binomial coefficient by
\[
\qbin{r}{k}=\frac{ [r]!}{[k]![r-k]!}\qquad \rm{for} \quad  r\geq k\geq 0.
\]
Nadeau and Tewari~\cite{NT21,NT23} introduced a $q$-analogue of the mixed Eulerian numbers $A_c(q)$, known as the \emph{remixed Eulerian numbers}. These numbers possess several remarkable properties, among which:
   \begin{enumerate}
       \item $A_c(q)$ is a polynomial with non-negative coefficients and $A_c(1)=A_c$,
       \item $A_{(c_1, \ldots, c_r)}(q)=q^{\binom{r}{2}}A_{(c_r, \ldots, c_1)}(q^{-1})$,
       \item $\sum_{c\in \mathcal{W}_r}A_c(q)=[r]! \cdot \frac{1}{r+1}\binom{2r}{r}$,
       \item $A_{(1, \ldots, 1)} (q)=[r]!$, and 
       \item $A_{(k,0, \ldots, 0, r-k)}(q)=q^{\binom{k}{2}}\cdot \qbin{r}{k}$.
   \end{enumerate}
   
   Recently, Gaudin~\cite{Ga24} observed that the Carlitz-Scoville numbers $A(r,s\,|\,\alpha,\beta)$ are a rescaled subfamily of the mixed Eulerian numbers and  proposed the   following $q$-analogue $A(r,s\,|\,\alpha,\beta)_q$ within the framework of the remixed Eulerian numbers,
\begin{equation}\label{def:solal}
    A(r,s\,|\,\alpha, \beta)_q=\frac{1}{[\alpha+\beta-1]!}A_{(0^r, 1^{\alpha-1}, r+s+1, 1^{\beta-1}, 0^s)}(q).
\end{equation}
  The following $q$-analogue of the recurrence~\eqref{rec-CS} is due to Gaudin~\cite[Proposition~4.3]{Ga24}:
\begin{equation}\label{rec-solal} 
A(r,s\,|\,\alpha,\beta)_q=q^{r+\beta-1}[s+\alpha]A(r-1,s\,|\,\alpha,\beta)_q+[r+\beta]A(r,s-1\,|\,\alpha,\beta)_q.
\end{equation}

For $\sigma=\sigma_1\sigma_2\ldots \sigma_n\in \Sym_n$, define  the following statistics
\begin{subequations}\label{def:statistics}
    \begin{align}  
    \des(\sigma)&=|\{i\in [n-1]\,|\, \sigma_i>\sigma_{i+1} \}|,\\
\lrmin(\sigma)&=|\{i\in [n]\,| \,\sigma_i<\sigma_{j},\; \forall j<i \}|,\\
\rlmin(\sigma)&=|\{i\in [n]\,|\, \sigma_i<\sigma_{j},\; \forall j>i \}|.
\end{align}
\end{subequations}
The statistic $\des$ is a well-known Eulerian statistic due to \eqref{eulerian polynomial}, both $\lrmin$ and 
$\rlmin$  are  Stirling statistics:
\begin{equation}
    \sum_{\sigma\in \Sym_n}x^{\lrmin(\sigma)}= \sum_{\sigma\in \Sym_n}x^{\rlmin(\sigma)}=\sum_{k=0}^ns(n,k)x^k,
\end{equation}
where $s(n,k)$ are the Stirling numbers of the first kind. The \emph{inversion number} $\inv$ and  \emph{major index} $\maj$ are defined by 
\begin{align}
    \inv(\sigma)&=|\{(i,j)\,|\, \sigma_i>\sigma_{j},\; i<j\}|,\\
\maj(\sigma)&=\sum_{\sigma_i>\sigma_{i+1}} i.
\end{align}
The statistics $\inv$ and $\maj$ are well-known Mahonian statistics:
\begin{equation}
    \sum_{\sigma\in \Sym_n}q^{\inv(\sigma)}= \sum_{\sigma\in \Sym_n}q^{\maj(\sigma)}=[n]!.
\end{equation}

For convenience, we introduce the following notation,  valid for $n\geq k\geq 0$:
 \begin{equation}\label{solal-CS}
    E_{n,k}(\alpha,\beta,q):=A(k,n-k\,|\,\alpha,\beta)_q
\end{equation}
 with  the convention $E_{n,k}(\alpha,\beta,q)=0$ if $k\notin \{0, \ldots, n\}$.  
 
 As $A_{(1^r)}(q) = [r]!$~(see \cite{NT23}), by \eqref{def:solal} we have $E_{0,0}(\alpha, \beta,q)=1$.
 In view of \eqref{solal-CS}, setting $r+s=n$ and $r=k$ in \eqref{rec-solal} we obtain
\begin{equation}\label{rec-a-butler}
        {E}_{n,k}(\alpha,\beta,q)=
        q^{\beta+k-1}[n-k+\alpha]{E}_{n-1,k-1}(\alpha,\beta,q)+[k+\beta]{E}_{n-1,k}(\alpha,\beta,q).
    \end{equation} 

 When $(\alpha,\beta)=(1,1)$ or $(0,1)$, several $q$-analogues of Eulerian numbers are already known~\cite{Ca54, St76, Raw81}. In the case $\alpha=0$, Butler~\cite{Bu04,Bu25} introduced a novel $q$-analogue of Eulerian numbers by combining the parameter $\beta$ with the major index. 
 
The aim of this paper is to provide a combinatorial interpretation of 
$E_{n,k}(\alpha, \beta, q)$ in terms of the permutation statistics such as
left-to-right minima, right-to-left minima, descents and (re)mixed major index, and  to derive  several explicit formulas for these numbers. Our results
 generalize and unify
 those of Carlitz-Scoville~\cite{CS74} and  Butler~\cite{Bu04}.
 
For $\sigma\in \Sym_n$ the \emph{mixed major index}  of $\sigma$, denoted by $\widetilde{\maj}$, is defined by
\begin{equation}\label{definitionE_n(a,b,q)-bis}
\widetilde{\maj}(\sigma)=(1+\mathrm{des}(\sigma)-\mathrm{lrmin}(\sigma))(\alpha-1)+{(n-\mathrm{rlmin}(\sigma))(\beta-1)+\mathrm{maj}(\sigma)}.
\end{equation}

\begin{theorem}\label{main theorem1}
    For integer $n\geq k\geq 0$, we have 
     \begin{equation}
  E_{n,k}(\alpha, \beta, q) = \sum\limits_{\sigma\in\Sym_{n+1,k}}[\alpha]^{\lrmin(\sigma)-1}
   [\beta]^{\rlmin(\sigma)-1}q^{\widetilde{\maj}(\sigma)}.
\end{equation}
where $\Sym_{n+1,k}$ denotes  the set of permutations in $\Sym_{n+1}$ with 
$k$ descents.
\end{theorem}

We define the \emph{bi-Stirling-Euler-Mahonian polynomials} $E_n(x\,|\,\alpha, \beta,q)$
by
\begin{equation} \label{rec-b-butler} 
E_n(x\,|\,\alpha, \beta,q)
=\sum_{k=0}^{n}E_{n,k}(\alpha,\beta,q)x^{k}.
\end{equation}
For example,  it follows from \eqref{rec-a-butler} that
\begin{align*}
E_0(x\,|\,\alpha, \beta,q)&=1,\\
E_1(x\,|\,\alpha, \beta,q)&= [\beta]+q^{\beta}[\alpha]x,\\
 E_2(x\,|\,\alpha, \beta,q)&= [\beta]^2+q^{\beta}([1+\alpha][\beta]+[\alpha][\beta+1])x+q^{2\beta+1}[\alpha]^2x^2.  
\end{align*}

 Define the $q$-derivative operator  $\delta_{x}$ by
\begin{equation}
   \delta_{x}\,p(x)=\frac{p(qx)-p(x)}{(q-1)x},
\end{equation}
where $p(x)\in \R[[x]]$.
 
  \begin{theorem} \label{main theorem2}
  Let $E_n(x;q):=E_n(x\,|\,\alpha, \beta, q) $ for $n\ge 0$. We have the following $q$-identities:
    \begin{itemize}
    \item[(i)] 
    \begin{align}\label{rec:XD}
        E_{n}(x;q)=
\bigl([\beta]+q^\beta[n-1+\alpha]x\bigr)E_{n-1}(x;q)+ (1-x)xq^{\beta}\delta_{x}(E_{n-1}(x;q)).
    \end{align}  

    \item[(ii)] 
    \begin{equation}\label{carlitz-E_n(x;q)}
     \frac{E_{n}(x;q)}{(x;q)_{n+\alpha+\beta}}=
\sum_{j\ge 0}x^j\qbin{j+\alpha+\beta-1}{j}[j+\beta]^n.
\end{equation}
    \item[(iii)] 
\begin{equation}\label{q-exp-GF}
    \sum_{n\ge 0}\frac{E_{n}(x;q)}{(x;q)_{n+\alpha+\beta}}\frac{t^n}{n!}
    =\sum_{j\ge 0}x^j\qbin{j+\alpha+\beta-1}{j}\exp([j+\beta]t).
\end{equation}
    \item[(iv)] 
\begin{equation}\label{worpitzky}
 \sum_{k=0}^{j}\qbin{j-k+n+\alpha+\beta-1}{j-k}E_{n,k}(\alpha,\beta,q)=  \qbin{j+\alpha+\beta-1}{j}[j+\beta]^n. 
\end{equation}  
    \item[(v)] \begin{equation}\label{explicit-formula-E}
    E_{n,k}(\alpha, \beta, q)=\sum_{j=0}^{k}\qbin{n+\alpha+\beta}{k-j}
    \qbin{j+\alpha+\beta-1}{j}(-1)^{k-j}
 q^{\binom{k-j}{2}}\,[\beta+j]^{n}.    
\end{equation}
\end{itemize}
\end{theorem}
We shall prove Theorem~\ref{main theorem1} and Theorem~\ref{main theorem2} in Section~2 and derive some corollaries in Section~3. 
Barbero et al.~\cite{BSV15} introduced the $\nu$-order $(s,t)$-Eulerian numbers. 
In Section~4 we show that when $\nu=1$ their  numbers coincide with the Carlitz-Scoville numbers
$A(r,s\,|\,\alpha, \beta)$.

\begin{remark}
    Using the inversion statistic on permutations, Dong-Lin-Pan~\cite{DLP24} obtained a $q$-analogue of \eqref{CS-egf} via Gessel’s $q$-compositional formula. For recent developments on Carlitz–Scoville’s generalized Eulerian polynomials, we refer the reader to~\cite{XZ24}.
    \end{remark}
\section{Proof of main Theorems}

\begin{proof}[\textbf{Proof of Theorem~\ref{main theorem1}}]
 Consider the enumerative polynomials
\begin{equation}\label{combina-defE_n(a,b,q)}
    \widehat{E}_{n,k}(\alpha,\beta,q)=\sum_{\sigma\in \Sym_{n+1,k}}[\alpha]^{\mathrm{lrmin}(\sigma)-1}[\beta]^{\mathrm{rlmin}(\sigma)-1}q^{\widetilde{\maj}(\sigma)}.
\end{equation} 
    It suffices to verify that the polynomials 
    $\widehat{E}_{n,k}(\alpha,\beta,q)$ satisfy the recurrence~\eqref{rec-a-butler}.
Clearly Eq.~\eqref{rec-a-butler} is valid for $n=0$. 
    It is clear that deleting $n+1$ from a permutation in $\Sym_{n+1, k}$, we obtain a permutation in $\Sym_{n, k}\cup \Sym_{n, k-1}$.   Inversely, any permutation $\pi$ in $\Sym_{n+1,k}$ can be created  by adding $n+1$ to a permutation $w$ in $\Sym_{n, k}\cup \Sym_{n, k-1}$.  We will discuss four types of insertion positions: a) at the beginning  of $w$,  b) at the end of $w$, c) at descent positions of $w$,  d) at ascent positions of $w$.  
 \begin{enumerate}
 \item[(a)] If   $w=\s_1\ldots \s_{n}\in \Sym_{n, k-1}$, then
  $\sigma=(n+1)\s_1\ldots \s_n\in \Sym_{n+1, k}$ with 
  \[(\lrmin, \rlmin, \des)\pi=(1+\lrmin, \rlmin, 1+\des)w.\] Hence
    \[
    \sum_{\substack{\pi\in \Sym_{n+1, k}\\ \pi(1)=n+1}}[\alpha]^{\mathrm{lrmin}(\pi)-1}[\beta]^{\mathrm{rlmin}(\pi)-1}q^{\widetilde{\maj}(\pi)}=[\alpha]q^{\beta+k-1}  \widehat{E}_{n-1,k-1}(\alpha,\beta,q).
    \]
    \item[(b)]
    If  $w=\s_1\ldots \s_{n}\in \Sym_{n, k}$, 
    then $\pi=\s_1\ldots \s_{n}(n+1)\in \Sym_{n+1, k}$
    and $(\lrmin, \rlmin)\pi=(\lrmin, 1+\rlmin)w$. Hence
    \[
    \sum_{\substack{\pi\in \Sym_{n+1, k}\\ \pi(n)=n+1}}[\alpha]^{\mathrm{lrmin}(\pi)-1}[\beta]^{\mathrm{rlmin}(\pi)-1}q^{\widetilde{\maj}(\pi)}=[\beta]  \widehat{E}_{n-1,k}(\alpha,\beta,q).
    \]
    
    \item[(c)] 
     If $w=\s_1\ldots \s_{n}\in \Sym_{n, k}$ with the descent set $\{i_1,i_2,\dots,i_{k}\}$,
     then   we define $\pi=\s_1\ldots \s_{i_j} (n+1) \s_{i_j+1}\ldots \s_{n}\in \Sym_{n+1,k}$ with $1\leq j\leq k$. It is clear that $(\lrmin, \rlmin)\pi=(\lrmin, \rlmin)w$ and
  $\widetilde{\maj}(\pi)=\widetilde{\maj}(w)+k-j$.
     Thus
    \[
    \sum_{\pi\in \Sym'_{n+1, k}}[\alpha]^{\mathrm{lrmin}(\pi)-1}[\beta]^{\mathrm{rlmin}(\pi)-1}q^{\widetilde{\maj}(\pi)}
    =
        q^{\beta}[k] \widehat{E}_{n-1,k}(\alpha,\beta,q),
    \]
    where  $\Sym'_{n+1, k}$ is the set of  permutations $\pi=\s_1\ldots \s_i (n+1) \s_{i+1}\ldots \s_{n}$ in $\Sym_{n+1, k}$ such that $\s_i>\s_{i+1}$.
  
    \item[(d)] 
     Let  $w=\s_1\ldots \s_{n}\in \Sym_{n, k-1}$
     with descent set $\{i_1, \ldots, i_{k-1}\}$.  If $\pi$ is a permutation 
     $\pi=\s_1\ldots \s_{j} (n+1) \s_{j+1}\ldots \s_{n}\in \Sym_{n+1, k}$  with $i_l<j<i_{l+1}$, where
     $0\leq l \leq k-1$, $i_0=0$ and $i_k=n$,  then 
     \begin{gather}
              (\lrmin, \rlmin, \des)\pi=(\lrmin, \rlmin, 1+\des)w,\\
       \widetilde{\maj}(\pi)
       =(k-1-l)+(j+1)+(\alpha-1)+(\beta-1)+\widetilde{\maj}(w).
        \end{gather}
Since
     \begin{align*}
     \bigcup_{l=0}^{k-1}\{k-1-l+j: i_l<j<i_{l+1}\}
     &=\bigcup_{l=0}^{k-1}\{k-l+i_l, \ldots, k-l+i_{l+1}-2\}\\
     &=\{k, k+1, \ldots, n-1\},
         \end{align*}
 we have 
      \[
\sum_{\substack{i_l<j<i_{l+1}\\ 0\leq l\leq k-1}}
     q^{\alpha+\beta-1+k-1-l+j}=q^{\alpha+\beta+k-1}[n-k].
      \]
It follows that 
     \begin{align*}
      \sum_{\pi\in \Sym^{''}_{n+1, k}} [\alpha]^{\lrmin(\pi)-1}
     [\beta]^{\rlmin(\pi)-1}
     q^{\widetilde{\maj}(\pi)}
     &=q^{\alpha+\beta+k-1}[n-k]\cdot 
     \widehat{E}_{n-1,k-1}(\alpha,\beta,q),
     \end{align*}
      where  $\Sym''_{n+1, k}$ is the  set of  permutations $\sigma=\sigma_1\ldots \sigma_i (n+1) \sigma_{i+1}\ldots \sigma_{n}$ in $\Sym_{n+1, k}$  such that $\sigma_i<\sigma_{i+1}$.
         \end{enumerate}
         Combining the two cases (a) and (d) (resp. (b) and (c))  we obtain 
     \begin{equation}\label{rec-a-butler-bis}
        \widehat{E}_{n,k}(\alpha,\beta,q)=
        q^{\beta+k-1}[n-k+\alpha]\widehat{E}_{n-1,k-1}(\alpha,\beta,q)+[k+\beta]\widehat{E}_{n-1,k}(\alpha,\beta,q).
    \end{equation} 
     Comparing the above equation with \eqref{rec-a-butler} we obtain  $\widehat{E}_{n,k}(\alpha,\beta,q)=E_{n,k}(\alpha,\beta,q)$.
     \end{proof}

For non-negative integer $n\in \Z$,
we define  $(x;q)_0=1$ and
\[
(x;q)_n=(1-x)(1-xq)\cdots(1-xq^{n-1})\qquad (n\geq 1).
\]
The $q$-binomial coefficients $\qbin{x}{n}$ are defined by
\[
\qbin{x}{n}=\frac{(q^{x-n+1};q)_n}{(q;q)_n}.
\]
Furthermore we  define
$(x;q)_\infty=\lim_{n\to \infty}(x;q)_n$
  and $(x;q)_n$ for all real numbers $n$ by
 \begin{equation}
 (x;q)_n=(x;q)_\infty/(xq^n;q)_\infty.
 \end{equation}
 Recall \cite[Chapter~2]{An76} that 
\begin{equation}
\sum_{n=0}^\infty\frac{(x;q)_nt^n}{(q;q)_n}=\frac{(xt;q)_\infty}{(t;q)_\infty}.
 \end{equation}
In particular we have 
\begin{align}\label{q-binomial-I}
(x;q)_N&=\sum_{i=0}^\infty \qbin{N}{i}(-1)^i x^i q^{{i(i-1)/2}},\\
\frac{1}{(x;q)_{N}}&=\sum_{j=0}^\infty x^{j}\qbin{N+j-1}{j}.\label{q-binomial-II}
\end{align}
For $k,n\in \N$ an easy calculation yields
\begin{equation}
    \delta_{x}\,\frac{x^k}{(x;q)_{n}}=x^{k-1}\frac{[k]+x[n-k]q^k}{(x;q)_{n+1}}=
    \begin{cases}
        [k]x^{k-1}, & \textrm{if} \;n=0;\\
        \frac{[n]}{(x;q)_{n+1}}, & \textrm{if}  \;k=0.
    \end{cases}
\end{equation}

\begin{proof}[\textbf{Proof of Theorem~\ref{main theorem2}}](i) 
     Multiplying Eq.~\eqref{rec-a-butler} or \eqref{rec-a-butler-bis} by $x^k$ and summing $k$ from $0$ to $n$ we obtain
     \begin{align}\label{eq:rec}
         \sum_{k= 0}^{n}E_{n,k}(\alpha,&\beta,q)x^k
         =\sum_{k= 0}^{n}\{q^{\beta+k-1}[n-k+\alpha]E_{n-1,k-1}(\alpha,\beta,q)+[k+\beta]E_{n-1,k}(\alpha,\beta,q)\}x^k\nonumber\\
         &=\sum_{k=0}^{n-1}q^{\beta+k}[n-1+\alpha-k]E_{n-1,k}(\alpha,\beta,q)x^{k+1}+\sum_{k=0}^{n-1}[k+\beta]E_{n-1,k}(\alpha,\beta,q)x^k.
     \end{align}
    Substituting $ [k+\beta]=[\beta]+q^{\beta}[k]$
     and  $ q^{\beta+k}[n-1+\alpha-k]=q^\beta[n-1+\alpha]-q^\beta[k]$ in \eqref{eq:rec} and then using $\delta_xx^{k}=[k]x^{k-1}$, we have
     \begin{align}\label{eq3}
         E_n(x;q)=&q^\beta [n-1+\alpha]xE_{n-1}(x;q)-q^\beta x^2\delta_x(E_{n-1}(x;q))\nonumber\\
         &+[\beta]E_{n-1}(x;q)+q^\beta x \delta_x(E_{n-1}(x;q)),
     \end{align}
     which is clearly equal to Eq.~\eqref{rec:XD}.
     
      (ii) We proceed by induction on $n\geq 0$. When $n=0$ Eq.~\eqref{carlitz-E_n(x;q)}  is 
  \eqref{q-binomial-II}  with $N=\alpha+\beta$.
  Multiplying \eqref{carlitz-E_n(x;q)} by $x^\beta$ and shifting $n\to n-1$
yields
\begin{align}\label{q-exp-equivalent-E_n(x;q)-n-1}
 \frac{x^{\beta} E_{n-1}(x;q)}{(x;q)_{n+\alpha+\beta-1}}=
 \sum_{j\ge 0}x^{j+\beta}\qbin{j+\alpha+\beta-1}{j}[j+\beta]^{n-1}.
\end{align}
 Applying   the operator $x^{1-\beta}\delta_x$ to 
 \eqref{q-exp-equivalent-E_n(x;q)-n-1} 
 and using  the $q$-Leibniz formula
\begin{equation}\label{product-rule-D_q}
    \delta_{x}\,(f(x)g(x))=\delta_{x}(f(x))g(qx)+f(x)\delta_{x}(g(x)),
\end{equation}
we obtain
\begin{gather}
\frac{(1-x)\big(q^{\beta}x\delta_{x}(E_n(x;q))+[\beta]E_n(x;q)\big)+[n+\alpha+\beta-1]xE_{n}(x;q)}{(x;q)_{n+\alpha+\beta}}\nonumber\\
  =\sum_{j\ge 0}x^{j}\qbin{j+\alpha+\beta-1}{j}[j+\beta]^{n}. 
\end{gather}
 Combining the above identity with \eqref{rec:XD}, we have~\eqref{carlitz-E_n(x;q)}.

 (iii) Multiplying Eq.~\eqref{carlitz-E_n(x;q)} by $t^n/n!$ followed by summing over $n\ge 0$, we have
 \begin{align}
     \sum_{n\ge 0}\frac{E_{n}(x;q)}{(x;q)_{n+\alpha+\beta}}\frac{t^n}{n!}
    &=\sum_{j\ge 0}x^j\qbin{j+\alpha+\beta-1}{j}\sum_{n\ge 0}\frac{([j+\beta]t)^n}{n!}\nonumber\\
    &=\sum_{j\ge 0}x^j\qbin{j+\alpha+\beta-1}{j}\exp([j+\beta]t).
 \end{align}

 (iv)  Eq.~\eqref{worpitzky} follows by extracting the coefficient of $x^j$ in both sides of Eq.~\eqref{carlitz-E_n(x;q)} and applying Eq.~\eqref{q-binomial-II}.

 (v) Eq.~\eqref{explicit-formula-E} follows from multiplying both sides of Eq.~\eqref{carlitz-E_n(x;q)} by $(x;q)_{n+\alpha+\beta}$ and extracting the coefficient of $x^k$ in both sides by using Eq.~\eqref{q-binomial-I}.
 \end{proof}

\section{Corollaries of main theorems}

In this section, we illustrate applications of our main theorem by specializing its parameters.

 \begin{coro}{\cite[Corollary~4.2.4]{Bu04-Phd}} For $n\ge 1$, we have 
\begin{equation}\label{beta-q-Carlitz}
     \frac{\sum_{\sigma\in \Sym_{n}}x^{\des(\sigma)}
q^{\widetilde{\maj_1}(\sigma)}[\beta]^{\rlmin(\sigma)}}{(x;q)_{n+\beta}} 
     =\sum_{j\geq 0}x^{j}\qbin{j+\beta-1}{j}
    [j+\beta]^{n},
\end{equation}
where $\widetilde{\maj_1}(\sigma)=(n-\rlmin(\sigma))(\beta-1)+\maj(\sigma)$.
\end{coro}
 \begin{proof} 
    When $\alpha=0$, the summation for $E_{n}(x\,|\, \alpha, \beta, q)$ reduces to the permutations $\sigma\in\Sym_{n+1}$ such that $\lrmin(\sigma)=1$, viz,  $\sigma_1=1$.
    Define  $\sigma'_i=\sigma_{i+1}-1$ for $i\in [n]$, 
    then  $\sigma'\in\Sym_n$  and 
    \begin{itemize}
    \item $\des(\sigma)=\des(\sigma')$;
        \item $\maj(\sigma)=\maj(\sigma')+\des(\sigma')$;
        \item $\rlmin(\sigma)=\rlmin(\sigma')+1$.
    \end{itemize}
    Hence, 
    \begin{align}\label{alpha=0}
        \sum_{\myatop{\sigma\in\Sym_{n+1}}{\lrmin(\sigma)=1}}x^{\des(\sigma)}q^{\widetilde{\maj}(\sigma)}[\beta]^{\rlmin(\sigma)-1}=\sum_{\sigma\in\Sym_{n}}x^{\des(\sigma)}q^{\widetilde{\maj_1}(\sigma)}[\beta]^{\rlmin(\sigma)}.
    \end{align}
   The result follows by combining~\eqref{carlitz-E_n(x;q)} and~\eqref{alpha=0}.
\end{proof}
\begin{coro}For $n\ge 1$, we have 
\begin{equation}\label{alpha-q-Carlitz}
     \frac{\sum_{\sigma\in \Sym_{n}}x^{\des(\sigma)}
q^{\widetilde{\maj_2}(\sigma)}[\alpha]^{\lrmin(\sigma)}}{(x;q)_{n+\alpha}} 
     =\sum_{j\geq 0}x^{j}\qbin{j+\alpha}{j+1}
    [j+1]^{n},
\end{equation}
where $\widetilde{\maj_2}(\sigma)=(1+\des(\sigma)-\lrmin(\sigma))(\alpha-1)+\maj(\sigma).$
\end{coro}
\begin{proof}
    When $\beta=0$, the summation for $E_{n}(x\,|\, \alpha, \beta, q)$ reduces to the permutations $\sigma\in\Sym_{n+1}$ with $\sigma_{n+1}=1$.
    Let  $\sigma'_i=\sigma_i-1$ for $i\in [n]$, 
    then $\sigma'\in\Sym_n$ and
    \begin{itemize}
    \item $\des(\sigma)=\des(\sigma')+1$;
        \item $\maj(\sigma)=\maj(\sigma')+n$;
        \item $\lrmin(\sigma)=\lrmin(\sigma')+1$.
    \end{itemize}
    Hence, 
    \begin{align}\label{beta=0}
        \sum_{\substack{\sigma\in\Sym_{n+1}\\\rlmin(\sigma)=1}}
        x^{\des(\sigma)}q^{\widetilde{\maj}(\sigma)}[\alpha]^{\lrmin(\sigma)-1}=\sum_{\sigma\in\Sym_{n}}x^{1+\des(\sigma)}q^{\widetilde{\maj_2}(\sigma)}[\alpha]^{\lrmin(\sigma)}.
    \end{align}
    The result follows by combining~\eqref{carlitz-E_n(x;q)} and~\eqref{beta=0}.
\end{proof}
Setting  $\alpha=1$ in \eqref{alpha-q-Carlitz} or $\beta=1$ in ~\eqref{beta-q-Carlitz}
we recover  Carlitz's identity~\cite{Ca54},
\begin{equation}    \frac{\sum\limits_{\sigma\in\Sym_n}x^{\des(\sigma)}q^{\maj(\sigma)}}{(x;q)_{n+1}}=\sum_{j\ge 0}x^j[j+1]^n.
    \end{equation} 
\begin{coro} For $n\geq 1$ we have 
\begin{equation}\label{sum of eulerian numbers} 
    E_{n}(1; q) =\prod_{i=0}^{n-1}[\alpha+\beta+i].
\end{equation}
\end{coro}
\begin{proof} 
   Note that 
   \[
     \frac{E_{n}(1,q)}{(q;q)_{n+\alpha+\beta-1}}=    \lim_{x\to 1^-} \frac{E_{n}(x,q)}{(xq;q)_{n+\alpha+\beta-1}}.
\]
Combining  \eqref{carlitz-E_n(x;q)} and Abel's lemma we have
          \begin{align*}
   E_{n}(1,q)&=(q;q)_{n+\alpha+\beta-1}
 \lim_{x\to 1^-}(1-x) \sum_{j\ge 0}x^j\qbin{j+\alpha+\beta-1}{j}[j+\beta]^n\\
 &=(q;q)_{n+\alpha+\beta-1}\lim_{j\to +\infty}\qbin{j+\alpha+\beta-1}{j}[j+\beta]^n\\
 &=\frac{(q;q)_{n+\alpha+\beta-1}}{(q;q)_{\a+\beta-1}(1-q)^n},
    \end{align*}
which is equivalent to 
 \eqref{sum of eulerian numbers}.
\end{proof}
\begin{remark}  
 Alternatively, setting $x=1$ in \eqref{rec:XD}  yields  
 $$E_{n}(1; q) =[n-1+\alpha+\beta]E_{n-1}(1;q).
 $$
 As $E_{0}(1; q)=1$,
we obtain  \eqref{sum of eulerian numbers} by iteration. Moreover, 
we derive  from \eqref{rec-a-butler} that
\begin{equation}
    E_{n,0}(\alpha,\beta,q)=[\beta]^n, \quad E_{n,n}(\alpha,\beta,q)=q^{n\beta+\binom{n}{2}}[\alpha]^n.
\end{equation}
\end{remark}
Let $\sigma=\sigma_1\sigma_2\dots\sigma_n\in\Sym_n$. 
For an integer $r\geq 1$,  as a generalization of descents, 
the set of \emph{r-descents} of $\sigma$  is defined by, see \cite{FS70},
\begin{equation}\label{r-des}
r\mathrm{Des}(\sigma):=\{i\in[n-1]:\sigma_i\ge\sigma_{i+1}+r\}.
\end{equation}
The cardinality of $r\mathrm{Des}(\sigma)$ is denoted by 
 $r\des(\sigma)$. 

 Rawlings~\cite{Raw81} introduced
the $r$-\emph{major index} as
\begin{equation}
    r\maj(\sigma):=\sum_{i\in r\mathrm{Des}(\sigma)}i+|r\mathrm{Inv}(\sigma)|,
\end{equation}
where 
\[r\mathrm{Inv}(\sigma):=\{(i,j): 1\leq i<j\leq n,\; \sigma_i>\sigma_j>\sigma_i-r\}.
\]
He then defined the $(q,r)$-Eulerian numbers
 \begin{equation}     A[n,k;r]:=\sum_{\substack{\sigma\in\Sym_n\\r\des(\sigma)=k}}q^{r\maj(\sigma)}.
 \end{equation}
 \begin{prop} For $r\ge 1$ and  $0\le k\le n$, we have
     \begin{equation}\label{r-maj:relation}
    A[n+r,k;r]=[r]!\cdot E_{n,k}(1,r,q).
\end{equation}
 \end{prop}
 \begin{proof}
     It is known~\cite{Raw81} that the $(q,r)$-Eulerian numbers satisfy the recurrence 
 \begin{equation}\label{recurrence-r-q-eulerian}
     A[n,k;r]=[k+r]A[n-1,k;r]+q^{k+r-1}[n+1-k-r]A[n-1,k-1;r],
 \end{equation}
 for $n\ge 1$, $0\le k\le n-r$ and  $A[r,0;r]=[r]!.$  Now, shifting  $n\to n+r$ Eq.~\eqref{recurrence-r-q-eulerian} becomes
\begin{equation}\label{rawlings r-maj recurrence}
     A[n+r,k;r]=[k+r]A[n+r-1,k;r]+q^{k+r-1}[n+1-k]A[n+r-1,k-1;r]
\end{equation}
 for $0\le k\le n$.
Comparing \eqref{rawlings r-maj recurrence} 
 with  recurrence~\eqref{rec-a-butler} and their initial  values at $(n,k)=(0,0)$, we derive \eqref{r-maj:relation}.
 \end{proof}
Combining Eq.~\eqref{carlitz-E_n(x;q)} and~\eqref{r-maj:relation}, we derive 
\begin{equation}~\label{q-hit} 
\frac{\sum_{0\le k\le n}A[n+r,k;r]x^k}{(x;q)_{n+1+r}}=\sum_{j\ge 0}x^j\frac{[j+r]!}{[j]!}[j+r]^n.
\end{equation}
From the above equality, we see that $A[n+r,k;r]$ for $0\le k\le n$ can be seen as $q$-hit numbers of Garsia-Remmel~\cite{GR86}.

\section{Concluding remarks}
In~\eqref{rec-a-butler}, let us define
\begin{equation}\label{E-star}
    E_{n,k}(\alpha,\beta,q)= 
    q^{k\beta +k(k-1)/2}\cdot E^{\star}_{n,k}(\alpha,\beta,q).
\end{equation}
Then 
\begin{equation}\label{rec-star}
        E^{\star}_{n,k}(\alpha,\beta,q)=
        [n-k+\alpha]E^{\star}_{n-1,k-1}(\alpha,\beta,q)+[k+\beta]E^{\star}_{n-1,k}(\alpha,\beta,q).
    \end{equation} 
It follows  that  
\begin{equation}\label{q-sym}
    E^{\star}_{n,k}(\alpha,\beta,q)=E^{\star}_{n,n-k}(\beta,\alpha,q) \qquad (0\le k\le n).
\end{equation}

Barbero et al.~\cite{BSV15} introduced a three-parameter  generalization of the standard Eulerian numbers $\angbinom{n}{k}^{(\nu)}_{(s, t)}$ for  integer $\nu\ge1$,  called  the  \emph{$\nu$-order $(s,t)$-Eulerian numbers}. When $\nu=1$, the numbers
$\angbinom{n}{k}_{(s,t)}:=\angbinom{n}{k}^{(1)}_{(s,t)}$ satisfy  the following recurrence  
\begin{equation}\label{eq:nu-st-eulerian}
  \angbinom{n}{k}_{(s,t)}
  = (k+s)\,\angbinom{n-1}{k}_{(s,t)}
    + \bigl(n - k + t\bigr)\,
      \angbinom{n-1}{k-1}_{(s,t)}+\delta_{k0}\delta_{n0},
\end{equation}
with the additional conditions 
\(
  \angbinom{n}{k}_{(s,t)} = 0
\)
if $n<0$ or $k<0$.
Comparing \eqref{eq:nu-st-eulerian} with~\eqref{rec-a-butler} (or \eqref{rec-star}) with $q=1$, we derive 
\begin{equation}\label{link of two types}
\angbinom{n}{k}_{(s,t)}=E_{n,k}(t,s,1).   
\end{equation}
\begin{remark}
From \eqref{euler's recurrence} and \eqref{eq:nu-st-eulerian} we derive
\begin{equation}
\angbinom{n}{k}=\angbinom{n}{k}_{(1,0)}, \quad 
\angbinom{n+1}{k}=\angbinom{n}{k}_{(1,1)}.
\end{equation}
Setting $q=1$ in Eq.~\eqref{carlitz-E_n(x;q)} we obtain
\begin{equation}\label{Eq-Carlitz-alpha-beta-S_n}
\frac{E_{n}(x\,|\,\alpha, \beta,1)}{(1-x)^{n+\alpha+\beta}}=
   \sum_{j\ge 0}x^j (j+\beta)^n\binom{\alpha+\beta+j-1}{j}.
\end{equation}
This is Proposition 12 in \cite{BSV15} by \eqref{link of two types}.  Setting $q=1$ in \eqref{q-exp-GF} we obtain 
\begin{equation}
    \sum_{n\geq 0}E_{n}(x\,|\,\alpha, \beta,1)\frac{z^n}{n!}=\left(\frac{1-x}{1-x e^{(1-x)z}}\right)^{\alpha}
    \left(\frac{1-x}{e^{(x-1)z}-x}\right)^{\beta},
\end{equation}
 which is   equivalent to~\eqref{CS-egf}.  
\end{remark}

The \emph{$r$-Eulerian numbers}, denoted by $A(n,k;r)$, count
the number of  permutations in $\Sym_{n}$ that have exactly  $k$ $r$-descents (see \eqref{r-des}), 
that is,
\begin{equation}
    A(n,k;r):=|\{\sigma\in\Sym_n:\mathrm{rdes}(\sigma)=k\}|.
\end{equation}
From~\eqref{rawlings r-maj recurrence} with $q=1$, we have 
\begin{equation}\label{r-eulerian}
     A(n+r,k;r)=(k+r)A(n+r-1,k;r)+(n+1-k)A(n+r-1,k-1;r)
\end{equation}
 for $0\le k\le n$.
Setting $(\alpha,\beta)=(r,1)$ in~\eqref{eq:nu-st-eulerian} and combining with~\eqref{r-eulerian} and their initial value at $(n,k)=(0,0)$, we derive
\begin{equation}\label{link-r-euler}
    \angbinom{n}{k}_{(r,1)}=\frac{1}{r!}A(n+r,k;r).
\end{equation}
 \begin{remark}
 In view of \eqref{link of two types} and \eqref{link-r-euler} 
 the $1$-order $(s,t)$-Eulerian numbers coincide with  the Carlitz-Scoville numbers and also encompass the $r$-Eulerian numbers. This fact was overlooked in \cite[p.2]{BSV15}, where  it is stated that  the Carlitz-Scoville numbers and the $r$-Eulerian numbers do not fall within  their general framework of Eulerian numbers.
 \end{remark}

The identities~\eqref{carlitz-E_n(x;q)},\eqref{worpitzky} and \eqref{q-hit}
deserve to  have combinatorial interpretations.
A combinatorial proof of \eqref{r-maj:relation}  with $q=1$ is given in \cite{FS70}. Moreover,
Butler~\cite{Bu04} originally obtained the permutation interpretation 
of the $q$-Eulerian numbers  $E_{n,k}(0,\beta,q)$ via the combinatorics of rook configurations (see also  Haglund~\cite{Ha96} and  Garsia and Remmel~\cite{GR86}), it would be  interesting  to extend our results to rook configurations.

\subsection*{Acknowledgement}
The first author is supported by the China Scholarship Council (No. 202206220034).




\bibliographystyle{plain}
\bibliography{reference2}


\end{document}